\documentclass[11pt,reqno]{amsart}

\usepackage[left=3.5cm,right=3.5cm,top=3.5cm,bottom=3.5cm]{geometry}
\usepackage{tipa}
\usepackage{amssymb}
\usepackage{graphicx}
\usepackage[hidelinks]{hyperref}
\usepackage{enumerate}
\usepackage{xcolor}
\usepackage{amsmath}
\allowdisplaybreaks[4]

\numberwithin{equation}{section}
\newtheorem{theorem}{Theorem}[section]

\newtheorem{corollary}{Corollary}[section]
\newtheorem{lemma}[theorem]{Lemma}

\theoremstyle{definition}

\newtheorem{remark}{Remark}[section]
\newtheorem*{remarks*}{Remarks}

\numberwithin{equation}{section}

\title{Perfect numbers and Fibonacci primes (III)}

\author[H. Zhong]{Hao Zhong}
\address{(H. Zhong) School of Mathematical Sciences, Zhejiang University, Hangzhou, 310027, China}
\curraddr{}
\email{11435011@zju.edu.cn}
\thanks{}

\author[T. Cai]{Tianxin Cai}
\address{(T. Cai) School of Mathematical Sciences, Zhejiang University, Hangzhou, 310027, China}
\curraddr{}
\email{txcai@zju.edu.cn}
\thanks{}

\keywords{perfect numbers, Lucas sequences, linear recurrent sequences}
\subjclass[2010]{11B83}

\begin{document}

\maketitle

\thispagestyle{empty}

\begin{abstract}
In this article, we consider the Diophantine equation $\sigma_{2}(n)-n^2=An+B$ with $A=P^2\pm2$. For some $B$, we show that except for finitely many computable solutions in the range $n\leq(|A|+|B|)^{3}$, all the solutions are expressible in terms of Lucas sequences. Meanwhile, we obtain some results relating to other linear recurrent sequences.
\end{abstract}

\section{Introduction}

It is well-known that a perfect number is a positive integer which equals to the sum of its proper divisors, namely, a positive integer $n$ is a perfect number if and only if
$$\sum_{d\mid n}d=2n.$$
The definition of perfect number is ancient, perhaps occurring in about 300 BC. Since then, many problems relating to perfect numbers have been studied by a number of great mathematicians including Fermat, Mersenne and Euler. And the Euclid-Euler theorem sates that every even perfect number can be represented by the form $2^{p-1}M_{p}$ where $M_{p}$ is Mersenne prime.

Let $\sigma_{k}(n)=\sum_{d\mid n}d^k$. Then a perfect number $n$ satisfies $\sigma_{1}(n)=2n$. In 2013, the second author raised the equation $\sigma_{2}(n)-n^{2}=3n$. It is proved in \cite{Cai2013} that all the solutions are $n=F_{2k-1}F_{2k+1}$ where both $F_{2k-1}$ and $F_{2k+1}$ are both Fibonacci primes. Later, Cai et al.\cite{Cai2014} studied a more generalized Diophantine equation $\sigma_{2}(n)-n^2=An+B$ where $A$ and $B$ are given integers and showed that if $(A, B)$ $\neq$ $(0,1)$ or $(1,1)$, then except for finitely many computable solutions in the range $n\leq(|A|+|B|)^{3}$, all the solutions are $n=pq$ with $p$, $q$ distinct primes.

In this paper, we continue to study $\sigma_{2}(n)-n^2$ for many other cases of $A$ and $B$. Given two integers $P$ and $Q$, the Lucas sequences of the first kind $U_{n}(P, Q)$ and of the second kind $V_{n}(P, Q)$ are defined by
$$U_{0}(P, Q)=0,$$
$$U_{1}(P, Q)=1,$$
$$U_{n}(P, Q)=P\cdot U_{n-1}(P, Q)-Q\cdot U_{n-2}(P, Q)\text{ for }n>1.$$
and
$$V_{0}(P, Q)=2,$$
$$V_{1}(P, Q)=P,$$
$$V_{n}(P, Q)=P\cdot V_{n-1}(P, Q)-Q\cdot V_{n-2}(P, Q)\text{ for }n>1.$$
Then, we obtain some interesting results.

\begin{theorem}\label{th:1}
Let $P$ be an integer. Then except for finitely many computable solutions in range $n\leq(|A|+|B|)^{3}$, all the solutions of $\sigma_{2}(n)-n^{2}=An+B$ are\\
(1) $n=U_{2k-1}(P,-1)U_{2k+1}(P,-1)$ with $U_{2k-1}(P,-1)$ and $U_{2k+1}(P,-1)$ primes if $A=P^{2}+2$ and $B=-P^{2}+1$;\\
(2) $n=U_{2k}(P,-1)U_{2k+2}(P,-1)$ with $U_{2k}(P,-1)$ and $U_{2k+2}(P,-1)$ primes if $A=P^{2}+2$ and $B=P^{2}+1$;\\
(3) $n=U_{k-1}(P,1)U_{k+1}(P,1)$ with $U_{k-1}(P,1)$ and $U_{k+1}(P,1)$ primes if $A=P^{2}-2$ and $B=P^{2}+1$.
\end{theorem}

Taking $P=1$ in (2), we can immediately obtain Theorem 2 in \cite{Cai2013}.
\begin{corollary}[Theorem 2, \cite{Cai2013}]
All the solutions of $\sigma_{2}(n)-n^2=3n$ are $n=F_{2k-1}F_{2k+1}$, where both $F_{2k-1}$ and $F_{2k+1}$ are Fibonacci primes.
\end{corollary}

And taking $P=2$ in (3), we can obtain the main result of \cite{Cai2014}.
\begin{corollary}[Theorem 5, {\cite{Cai2014}}]
Twin primes conjecture holds if and only if the equation $\sigma_{2}(n)-n^2=2n+5$ has infinitely many solutions.
\end{corollary}

As for the second kind of Lucas sequences, we also have some results.
\begin{theorem}\label{th:2}
Let $P$ be an integer. Then except for finitely many computable solutions in range $n\leq(|A|+|B|)^{3}$, all the solutions of $\sigma_{2}(n)-n^{2}=An+B$ are\\
(1) $n=V_{2k}(P,-1)V_{2k+2}(P,-1)$ with $V_{2k}(P,-1)$ and $V_{2k+2}(P,-1)$ primes if $P^2+4$ is square-free, $A=P^{2}+2$ and $B=-P^{4}-4P^{2}+1$;\\
(2) $n=V_{2k-1}(P,-1)V_{2k+1}(P,-1)$ with $V_{2k-1}(P,-1)$ and $V_{2k+1}(P,-1)$ primes if $P^2+4$ is square-free, $A=P^{2}+2$ and $B=P^{4}+4P^{2}+1$;\\
(3) $n=V_{k-1}(P,1)V_{k+1}(P,1)$ with $V_{k-1}(P,1)$ and $V_{k+1}(P,1)$ primes if $P^2-4$ is square-free, $A=P^{2}-2$ and $B=-P^{4}+4P^{2}+1$.
\end{theorem}

Let $m$ be any integer. For some other special pairs $(A,B)$ relating to Lucas sequences, we have
\begin{theorem}\label{th:3}
Let $P$ be an integer. Then except for finitely many computable solutions in range $n\leq(|V_{2m}(P,-1)|+U_{2m}^{2}(P,-1)-1)^{3}$, all the solutions of
\begin{equation}\label{eq:th:3}
\sigma_{2}(n)-n^{2}=V_{2m}(P,-1)n-U_{2m}^{2}(P,-1)+1
\end{equation}
are\\
(1) $n=U_{2k+1}(P,-1)U_{2k+2m+1}(P,-1)$ with $U_{2k+1}(P,-1)$ and $U_{2k+2m+1}(P,-1)$ primes;\\
(2) $n=U_{2k+1}(P,-1)U_{2m-2k-1}(P,-1)$($m\neq 2k+1$) with $U_{2k}(P,-1)$ and $U_{2m-2k-1}(P,-1)$ primes.
\end{theorem}

\begin{theorem}\label{th:4}
Let $P$ be an integer. Then except for finitely many computable solutions in range $n\leq(|V_{2m}(P,-1)|+U_{2m}^{2}(P,-1)+1)^{3}$, all the solutions of
\begin{equation}\label{eq:th:4}
\sigma_{2}(n)-n^{2}=V_{2m}(P,-1)n+U_{2m}^{2}(P,-1)+1
\end{equation}
are\\
(1) $n=U_{2k}(P,-1)U_{2k+2m}(P,-1)$ with $U_{2k}(P,-1)$ and $U_{2k+2m}(P,-1)$ primes;\\
(2) $n=U_{2k}(P,-1)U_{2m-2k}(P,-1)$($m\neq 2k$) with $U_{2k}(P,-1)$ and $U_{2m-2k}(P,-1)$ primes.
\end{theorem}

\begin{theorem}\label{th:5}
Let $P$ be an integer and $P^{2}+4$ is square-free. Then except for finitely many computable solutions in range $n\leq(|V_{2m}(P,-1)|+V_{2m}^{2}(P,-1)-3)^{3}$, all the solutions of
\begin{equation}\label{eq:th:5}
\sigma_{2}(n)-n^{2}=V_{2m}(P,-1)n+V_{2m}^{2}(P,-1)-3
\end{equation}
are\\
(1) $n=V_{2k+1}(P,-1)V_{2k+2m+1}(P,-1)$ with $V_{2k+1}(P,-1)$ and $V_{2k+2m+1}(P,-1)$ primes;\\
(2) $n=V_{2k+1}(P,-1)V_{2m-2k-1}(P,-1)$($m\neq 2k+1$) with $V_{2k+1}(P,-1)$ and $V_{2m-2k-1}(P,-1)$ primes.
\end{theorem}

\begin{theorem}\label{th:6}
Let $P$ be an integer and $P^{2}+4$ is square-free. Then except for finitely many computable solutions in range $n\leq(|V_{2m}(P,-1)|+V_{2m}^{2}(P,-1)-5)^{3}$, all the solutions of
\begin{equation}\label{eq:th:6}
\sigma_{2}(n)-n^{2}=V_{2m}(P,-1)n-V_{2m}^{2}(P,-1)+5
\end{equation}
are\\
(1) $n=V_{2k}(P,-1)V_{2k+2m}(P,-1)$ with $V_{2k}(P,-1)$ and $V_{2k+2m}(P,-1)$ primes;\\
(2) $n=V_{2k}(P,-1)V_{2m-2k}(P,-1)$($m\neq 2k$) with $V_{2k}(P,-1)$ and $V_{2m-2k}(P,-1)$ primes.
\end{theorem}

Taking $P=1$ in these four theorems, one can obtain the results in \cite{Cai2014}.
\begin{corollary}[{\cite[Theorem 3]{Cai2014}}]
Except for finitely many computable solutions in the range $n\leq (L_{2m}+F_{2m}^{2}-1)^{3}$, all the solutions of
$$\sigma_{2}(n)-n^{2}=L_{2m}n-F_{2m}^{2}+1$$
are\\
(1) $n=F_{2k+1}F_{2k+2m+1}$($k\geq 0$) with $F_{2k+1}$ and $F_{2k+2m+1}$ Fibonacci primes;\\
(2) $n=F_{2k+1}F_{2m-2k-1}$($0\leq k<m$, $m\neq 2k+1$) with $F_{2k+1}$ and $F_{2m-2k-1}$ Fibonacci primes.
\end{corollary}

\begin{corollary}
Except for finitely many computable solutions in the range $n\leq (L_{2m}+F_{2m}^{2}+1)^{3}$, all the solutions of
$$\sigma_{2}(n)-n^{2}=L_{2m}n+F_{2m}^{2}+1$$
are\\
(1) $n=F_{2k}F_{2k+2m}$($k\geq 0$) with $F_{2k}$ and $F_{2k+2m}$ Fibonacci primes;\\
(2) $n=F_{2k}F_{2m-2k}$($0\leq k<m$, $m\neq 2k$) with $F_{2k}$ and $F_{2m-2k}$ Fibonacci primes.
\end{corollary}

\begin{corollary}[{\cite[Theorem 4]{Cai2014}}]
Except for finitely many computable solutions in the range $n\leq (L_{2m}+L_{2m}^{2}-3)^{3}$, all the solutions of
$$\sigma_{2}(n)-n^{2}=L_{2m}n+L_{2m}^{2}-3$$
are $n=L_{2k-1}L_{2k+2m-1}$ with $F_{2k-1}$ and $F_{2k+2m-1}$ Lucas primes.
\end{corollary}

\begin{corollary}[{\cite[Theorem 5]{Cai2014}}]
Except for finitely many computable solutions in the range $n\leq (L_{2m}+L_{2m}^{2}-5)^{3}$, all the solutions of
$$\sigma_{2}(n)-n^{2}=L_{2m}n-L_{2m}^{2}+5$$
are\\
(1) $n=L_{2k+1}L_{2k+2m}$($k\geq 0$) with $L_{2k+1}$ and $L_{2k+2m+1}$ Lucas primes;\\
(2) $n=L_{2k+1}L_{2m-2k}$($0\leq k<m$, $m\neq 2k$) with $L_{2k+1}$ and $L_{2m-2k-1}$ Lucas primes.
\end{corollary}

Finally, we consider a conjecture that mentioned in \cite{Cai2013}, also seen in \cite{Cai}, that is,\\
\textbf{Conjecture 1.} $\omega(n)$=2 and $n\mid \sigma_{3}(n)$ iff $n$ is an even perfect number except for 28.
where $\omega(n)$ denotes the number of distinct prime factors of $n$.

In \cite{Cai2013}, Cai et al. prove that it is true for $n=pq$ and $2^{\alpha}p$. In this article, we obtain a result closer to the statement of this conjecture, namely,
\begin{theorem}\label{th:7}
Let $n=pq^{\alpha}$ $(\alpha\geq1)$ with $p$, $q$ distinct primes and $q\not\equiv1\pmod{3}$. If $n\mid \sigma_{3}(n)$, then $n$ is an even perfect number except for 28.
\end{theorem}

\section{Preliminaries}

In order to prove the theorems, we need some lemmas.

\begin{lemma}[{\cite[Lemma 2.4]{Zhong2016}}]\label{lem:1}
Let $A_{n}$ be a linear recurrent sequence, i.e., $\{A_{n}\}$ satisfies the linear recurrent relation:
$$A_{n}=uA_{n-1}+vA_{n-2}\text{ for }n\geq2$$
where $A_{0}$, $A_{1}$, $u$ and $v$ are given integers. Then
\begin{equation}\label{eq:lem1}
A_{n+r}A_{n-r}-A_{n}^{2}=(-v)^{n-r}s^{2}(r-1,u,v)(vA_{0}^{2}+uA_{0}A_{1}-A_{1}^{2})
\end{equation}
where $s(k,u,v)=\sum_{i=0}^{\lfloor \frac{k}{2} \rfloor}\binom{k-i}{i}u^{k-2i}v^{i}$.
\end{lemma}

Taking $r=1$ in Lemma \ref{lem:1}, we can obtain the following equations.

If $v=1$, then
\begin{equation}\label{eq:lem1.1}
1+A_{2n}^{2}+A_{2n+2}^{2}=(u^{2}+2)A_{2n}A_{2n+2}-u^{2}(A_{0}^{2}+uA_{0}A_{1}-A_{1}^{2})+1;
\end{equation}

\begin{equation}\label{eq:lem1.2}
1+A_{2n-1}^{2}+A_{2n+1}^{2}=(u^{2}+2)A_{2n-1}A_{2n+1}+u^{2}(A_{0}^{2}+uA_{0}A_{1}-A_{1}^{2})+1.
\end{equation}

And if $v=-1$, then
\begin{equation}\label{eq:lem1.3}
1+A_{n-1}^{2}+A_{n+1}^{2}=(u^{2}-2)A_{n-1}A_{n+1}+u^{2}(A_{0}^{2}-uA_{0}A_{1}+A_{1}^{2})+1.
\end{equation}

These equations can readily prove that all the $n=pq$ in Theorem \ref{th:1} and \ref{th:2} are solutions to $\sigma_{2}(n)-n^{2}=An+B$.

\begin{lemma}[{\cite[Theorem 2]{Cai2014}}]\label{lem:2}
If $(A, B)$ $\neq$ $(0,1)$ or $(1,1)$, then except for finitely many computable solutions in the range $n\leq(|A|+|B|)^{3}$, all the solutions of $\sigma_{2}(n)-n^{2}=An+B$ are $n=pq$, where $p<q$ are primes which satisfy the following equation
\begin{equation}\label{eq:lem2.1}
p^{2}+q^{2}+1-B=Apq,
\end{equation}
or equally,
\begin{equation}\label{eq:lem2.2}
(2p-Aq)^{2}-(A^{2}-4)q^{2}=4(B-1).
\end{equation}
\end{lemma}

\begin{lemma}[{Matiyasevich equation\cite{Matiyasevich1970}}]\label{lem:3}
All the solutions of $x^{2}-(P^{2}+4)y^{2}=4$ are given by $x=V_{2k}(P,-1)$ and $y=U_{2k}(P,-1)$.
\end{lemma}

\begin{remark}
The original version of Matiyasevich equation requires $P$ to be a positive integer. Nevertheless, it is easy to check that Lemma \ref{lem:3} also holds for $P\leq0$.
\end{remark}

\begin{lemma}[{\cite[Corollary 2.8]{Jones2003}}]\label{lem:4}
All the solutions of $x^{2}-(P^{2}+4)y^{2}=-4$ are given by $x=V_{2k+1}(P,-1)$ and $y=U_{2k+1}(P,-1)$.
\end{lemma}

\begin{lemma}[{\cite[Corollary 2.7]{Jones2003}}]\label{lem:5}
All the solutions of $x^{2}-(P^{2}-4)y^{2}=4$ are given by $x=V_{k}(P,1)$ and $y=U_{k}(P,1)$.
\end{lemma}

\begin{remark}
If $P=3$, then this lemma can be proved by applying Lemma \ref{lem:4} since $P^{2}-4=5=1^{2}+4$.  And for $P\leq2$, it is easy to check. Hence the condition $P\geq4$ in \cite{Jones2003} can be removed.
\end{remark}

The following lemma is easy to check.

\begin{lemma}\label{lem:6}
Let $k$ and $l$ be any integer. The terms of Lucas sequences satisfy the following relations.

(1) $V_{k}(P,Q)=U_{k+1}(P,Q)-QU_{k-1}(P,Q)$.

(2) $(P^{2}-4Q)U_{k}(P,Q)=V_{k+1}(P,Q)-QV_{k-1}(P,Q)$.

(3) $2U_{k+l}(P,Q)=U_{k}(P,Q)V_{l}(P,Q)+U_{l}(P,Q)V_{k}(P,Q)$.

(4) $2Q^{l}U_{k-l}(P,Q)=U_{k}(P,Q)V_{l}(P,Q)-U_{l}(P,Q)V_{k}(P,Q)$.

(5) $2V_{k+l}(P,Q)=(P^{2}-4Q)U_{k}(P,Q)U_{l}(P,Q)+V_{k}(P,Q)V_{l}(P,Q)$.

(6) $2Q^{l}V_{k+l}(P,Q)=(P^{2}-4Q)U_{k}(P,Q)U_{l}(P,Q)-V_{k}(P,Q)V_{l}(P,Q)$.
\end{lemma}

At last, to prove Theorem \ref{th:7}, we need the following lemma.

\begin{lemma}[{\cite[Theorem 4]{Cai2013}}]\label{lem:7}
If $n=pq$, where $p<q$ are primes and $n\mid \sigma_{3}(n)$, then $n=6$; if $n=2^{\alpha}p$ $(\alpha \geq1)$, $p$ is an odd prime and $n\mid \sigma_{3}(n)$, then $n$ is an even perfect number. The converse is also true except for 28.
\end{lemma}

\section{Proofs of the Theorem 1.1 -- 1.6}

From Lemma \ref{lem:2}, we know that except for finitely many computable solutions in the range $n\leq(|A|+|B|)^{3}$, all the solutions of $\sigma_{2}(n)-n^{2}=An+B$ are $n=pq$, where $p<q$ are primes which satisfy the equation (\ref{eq:lem2.2}). Hence, to prove Theorem \ref{th:1} to \ref{th:6}, we only need to consider the case that $n$ is a product of two different primes $p$ and $q$.

\begin{proof}[Proof of Theorem \ref{th:1}]

(1) Taking $(A,B)=(P^{2}+2,-P^{2}+1)$ in (\ref{eq:lem2.2}), then
$$(2p-(P^{2}+2)q)^{2}-P^{2}(P^{2}+4)q^{2}=-4P^{2}.$$
Let $r:=\frac{2p-(P^{2}+2)q}{P}$. Then $r$ is an integer satisfying that $r^{2}-(P^{2}+4)q^{2}=-4$. By Lemma \ref{lem:4}, we have $(r,q)=(\pm V_{2k+1}(P,-1),U_{2k+1}(P,-1))$ for some integer $k$. By (1) of Lemma \ref{lem:6}, similarly, we have $p=U_{2k-1}(P,-1)$ or $U_{2k+3}(P,-1)$. Hence $n=U_{2k-1}(P,-1)U_{2k+1}(P,-1)$ with $U_{2k-1}(P,-1)$ and $U_{2k+1}(P,-1)$ primes.

(2) Taking $(A,B)=(P^{2}+2,P^{2}+1)$ in (\ref{eq:lem2.2}), then
$$(2p-(P^{2}+2)q)^{2}-P^{2}(P^{2}+4)q^{2}=4P^{2}.$$
Let $r:=\frac{2p-(P^{2}+2)q}{P}$. Then $r$ is an integer satisfying that $r^{2}-(P^{2}+4)q^{2}=4$. By Lemma \ref{lem:3}, we have $(r,q)=(\pm V_{2k}(P,-1),U_{2k}(P,-1))$ for some integer $k$. By (1) of Lemma \ref{lem:6},

\textbf{Case 1.} $(r,q)=(V_{2k}(P,-1),U_{2k}(P,-1))$, then
\begin{align*}
p &= \{rP+(P^{2}+2)q\}/2\\
&= \{PV_{2k}(P,-1)+(P^{2}+2)U_{2k}(P,-1)\}/2\\
&= \{PU_{2k-1}(P,-1)+PU_{2k+1}(P,-1)+(P^{2}+2)U_{2k}(P,-1)\}/2\\
&= \{P(U_{2k+1}(P,-1)-PU_{2k}(P,-1))+PU_{2k+1}(P,-1)+(P^{2}+2)U_{2k}(P,-1)\}/2\\
&= \{2PU_{2k+1}(P,-1)+2U_{2k}(P,-1)\}/2\\
&= U_{2k+2}(P,-1).
\end{align*}

\textbf{Case 2.} $(r,q)=(-V_{2k}(P,-1),U_{2k}(P,-1))$, then
\begin{align*}
p &= \{rP+(P^{2}+2)q\}/2\\
&= \{-PV_{2k}(P,-1)+(P^{2}+2)U_{2k}(P,-1)\}/2\\
&= \{-PU_{2k-1}(P,-1)-PU_{2k+1}(P,-1)+(P^{2}+2)U_{2k}(P,-1)\}/2\\
&= \{-PU_{2k-1}(P,-1)-P(PU_{2k}(P,-1)+U_{2k-1}(P,-1))+(P^{2}+2)U_{2k}(P,-1)\}/2\\
&= \{2U_{2k}(P,-1)-2PU_{2k-1}(P,-1)\}/2\\
&= U_{2k-2}(P,-1).
\end{align*}
Hence $n=U_{2k}(P,-1)U_{2k+2}(P,-1)$ with $U_{2k}(P,-1)$ and $U_{2k+2}(P,-1)$ primes.

(3) Taking $(A,B)=(P^{2}-2,P^{2}+1)$ in (\ref{eq:lem2.2}), then
$$(2p-(P^{2}-2)q)^{2}-P^{2}(P^{2}-4)q^{2}=4P^{2}.$$
Let $r:=\frac{2p-(P^{2}-2)q}{P}$. Then $r$ is an integer satisfying that $r^{2}-(P^{2}-4)q^{2}=4$. By Lemma \ref{lem:5}, we have $(r,q)=(\pm V_{k}(P,1),U_{k}(P,1))$ for some integer $k$. By (1) of Lemma \ref{lem:6}, similarly, we have $p=U_{k-2}(P,1)$ or $U_{k+2}(P,1)$. Hence $n=U_{k-1}(P,1)U_{k+1}(P,1)$ with $U_{k-1}(P,1)$ and $U_{k+1}(P,1)$ primes.
\end{proof}

\begin{proof}[Proof of Theorem \ref{th:2}]

(1) Taking $(A,B)=(P^{2}+2,-P^{4}-4P^{2}+1)$ in (\ref{eq:lem2.2}), then
$$(2p-(P^{2}+2)q)^{2}-P^{2}(P^{2}+4)q^{2}=-4P^{2}(P^{2}+4).$$
Let $r:=\frac{2p-(P^{2}+2)q}{P(P^{2}+4)}$. Then $r$ is an integer satisfying that $q^{2}-(P^{2}+4)r^{2}=4$ since $P^{2}+4$ is square-free. By Lemma \ref{lem:3}, we have $(r,q)=(\pm U_{2k}(P,-1),V_{2k}(P,-1))$ for some integer $k$. By (2) of Lemma \ref{lem:6},

\textbf{Case 1.} $(r,q)=(U_{2k}(P,-1),V_{2k}(P,-1))$, then
\begin{align*}
p &= \{rP(P^{2}+4)+(P^{2}+2)q\}/2\\
&= \{P(P^{2}+4)U_{2k}(P,-1)+(P^{2}+2)V_{2k}(P,-1)\}/2\\
&= \{PV_{2k-1}(P,-1)+PV_{2k+1}(P,-1)+(P^{2}+2)V_{2k}(P,-1)\}/2\\
&= \{P(V_{2k+1}(P,-1)-PV_{2k}(P,-1))+PV_{2k+1}(P,-1)+(P^{2}+2)V_{2k}(P,-1)\}/2\\
&= \{2PV_{2k+1}(P,-1)+2V_{2k}(P,-1)\}/2\\
&= V_{2k+2}(P,-1).
\end{align*}

\textbf{Case 2.} $(r,q)=(-U_{2k}(P,-1),V_{2k}(P,-1))$, then
\begin{align*}
p &= \{rP(P^{2}+4)+(P^{2}+2)q\}/2\\
&= \{-P(P^{2}+4)U_{2k}(P,-1)+(P^{2}+2)V_{2k}(P,-1)\}/2\\
&= \{-PV_{2k-1}(P,-1)-PV_{2k+1}(P,-1)+(P^{2}+2)V_{2k}(P,-1)\}/2\\
&= \{-PV_{2k-1}(P,-1)-P(PV_{2k}(P,-1)+V_{2k-1}(P,-1))+(P^{2}+2)V_{2k}(P,-1)\}/2\\
&= \{2V_{2k}(P,-1)-2PV_{2k-1}(P,-1)\}/2\\
&= V_{2k-2}(P,-1).
\end{align*}
Hence $n=V_{2k}(P,-1)V_{2k+2}(P,-1)$ with $V_{2k}(P,-1)$ and $V_{2k+2}(P,-1)$ primes.

(2) Taking $(A,B)=(P^{2}+2,P^{4}+4P^{2}+1)$ in (\ref{eq:lem2.2}), then
$$(2p-(P^{2}+2)q)^{2}-P^{2}(P^{2}+4)q^{2}=4P^{2}(P^{2}+4).$$
Let $r:=\frac{2p-(P^{2}+2)q}{P(P^{2}+4)}$. Then $r$ is an integer satisfying that $q^{2}-(P^{2}+4)r^{2}=-4$ since $P^{2}+4$ is square-free. By Lemma \ref{lem:4}, we have $(r,q)=(\pm U_{2k+1}(P,-1),V_{2k+1}(P,-1))$ for some integer $k$. By (2) of Lemma \ref{lem:6}, similarly, we have $p=V_{2k-1}(P,-1)$ or $V_{2k+3}(P,-1)$. Hence $n=V_{2k-1}(P,-1)V_{2k+1}(P,-1)$ with $V_{2k-1}(P,-1)$ and $V_{2k+1}(P,-1)$ primes.

(3) Taking $(A,B)=(P^{2}-2,-P^{4}+4P^{2}+1)$ in (\ref{eq:lem2.2}), then
$$(2p-(P^{2}-2)q)^{2}-P^{2}(P^{2}-4)q^{2}=-4P^{2}(P^{2}-4).$$
Let $r:=\frac{2p-(P^{2}-2)q}{P(P^{2}-4)}$. Then $r$ is an integer satisfying that $q^{2}-(P^{2}-4)r^{2}=4$ since $P^{2}-4$ is square-free. By Lemma \ref{lem:5}, we have $(r,q)=(\pm U_{k}(P,1),V_{k}(P,1))$ for some integer $k$. By (2) of Lemma \ref{lem:6}, similarly, we have $p=V_{k-2}(P,1)$ or $V_{k+2}(P,1)$. Hence $n=V_{k-1}(P,1)V_{k+1}(P,1)$ with $V_{k-1}(P,1)$ and $V_{k+1}(P,1)$ primes.

\end{proof}

\begin{proof}[Proof of Theorem \ref{th:3}]

Taking $(A,B)=(V_{2m}(P,-1),-U_{2m}^{2}(P,-1)+1)$ in (\ref{eq:lem2.2}), then
$$(2p-V_{2m}(P,-1)q)^{2}-(V_{2m}^{2}(P,-1)-4)q^{2}=-4U_{2m}^{2}(P,-1).$$
And by Lemma \ref{lem:3}, we have
$$(2p-V_{2m}(P,-1)q)^{2}-(P^{2}+4)U_{2m}^{2}(P,-1)q^{2}=-4U_{2m}^{2}(P,-1).$$
Let $r:=\frac{2p-V_{2m}(P,-1)q}{U_{2m}(P,-1)}$. Then $r$ is an integer satisfying that $r^{2}-(P^{2}+4)q^{2}=-4$. By Lemma \ref{lem:4}, we have $(r,q)=(\pm V_{2k+1}(P,-1),U_{2k+1}(P,-1))$ for some integer $k$. By (3) and (4) of Lemma \ref{lem:6},

\textbf{Case 1.} $(r,q)=(V_{2k+1}(P,-1),U_{2k+1}(P,-1))$, then
\begin{align*}
p &= \{rU_{2m}(P,-1)+V_{2m}(P,-1)q\}/2\\
&= \{V_{2k+1}(P,-1)U_{2m}(P,-1)+V_{2m}(P,-1)U_{2k+1}(P,-1)\}/2\\
&= U_{2k+2m+1}(P,-1).
\end{align*}

\textbf{Case 2.} $(r,q)=(-V_{2k+1}(P,-1),U_{2k+1}(P,-1))$, then
\begin{align*}
p &= \{rU_{2m}(P,-1)+V_{2m}(P,-1)q\}/2\\
&= \{-V_{2k+1}(P,-1)U_{2m}(P,-1)+V_{2m}(P,-1)U_{2k+1}(P,-1)\}/2.
\end{align*}
If $2m\geq 2k+1$, then $p=U_{2m-2k-1}(P,-1)$. Otherwise, $p=U_{2k+1-2m}(P,-1)$. Hence, $n=U_{2k+1}(P,-1)U_{2k+2m+1}(P,-1)$ with $U_{2k+1}(P,-1)$ and $U_{2k+2m+1}(P,-1)$ primes; or $U_{2k+1}(P,-1)U_{2m-2k-1}(P,-1)$($m\neq 2k+1$) with $U_{2k+1}(P,-1)$ and $U_{2m-2k-1}(P,-1)$ primes.
\end{proof}

\begin{proof}[Proof of Theorem \ref{th:4}]

Taking $(A,B)=(V_{2m}(P,-1),U_{2m}^{2}(P,-1)+1)$ in (\ref{eq:lem2.2}), then
$$(2p-V_{2m}(P,-1)q)^{2}-(V_{2m}^{2}(P,-1)-4)q^{2}=4U_{2m}^{2}(P,-1).$$
And by Lemma \ref{lem:3}, we have
$$(2p-V_{2m}(P,-1)q)^{2}-(P^{2}+4)U_{2m}^{2}(P,-1)q^{2}=4U_{2m}^{2}(P,-1).$$
Let $r:=\frac{2p-V_{2m}(P,-1)q}{U_{2m}(P,-1)}$. Then $r$ is an integer satisfying $r^{2}-(P^{2}+4)q^{2}=4$. From Lemma \ref{lem:3}, we have $(r,q)=(\pm V_{2k}(P,-1),U_{2k}(P,-1))$ for some integer $k$. By (3) and (4) of Lemma \ref{lem:6}, similarly, we can obtain $n=U_{2k}(P,-1)U_{2k+2m}(P,-1)$ with $U_{2k}(P,-1)$ and $U_{2k+2m}(P,-1)$ primes; or $U_{2k}(P,-1)U_{2m-2k}(P,-1)$($m\neq 2k$) with $U_{2k}(P,-1)$ and $U_{2m-2k}(P,-1)$ primes.
\end{proof}

\begin{proof}[Proof of Theorem \ref{th:5}]

Taking $(A,B)=(V_{2m}(P,-1),V_{2m}^{2}(P,-1)-3)$ in (\ref{eq:lem2.2}), then
$$(2p-V_{2m}(P,-1)q)^{2}-(V_{2m}^{2}(P,-1)-4)q^{2}=4(V_{2m}^{2}(P,-1)-4).$$
And by Lemma \ref{lem:3}, we have
$$(2p-V_{2m}(P,-1)q)^{2}-(P^{2}+4)U_{2m}^{2}(P,-1)q^{2}=4(P^{2}+4)U_{2m}^{2}(P,-1).$$
Let $r:=\frac{2p-V_{2m}(P,-1)q}{(P^{2}+4)U_{2m}(P,-1)}$. Then $r$ is an integer satisfying that $q^{2}-(P^{2}+4)r^{2}=-4$ since $P^{2}+4$ is square-free. From Lemma \ref{lem:4}, we have $(r,q)=(\pm U_{2k+1}(P,-1),V_{2k+1}(P,-1))$ for some integer $k$. By (5) and (6) of Lemma \ref{lem:6}, we take it into two cases.

\textbf{Case 1.} If $(r,q)=(U_{2k+1}(P,-1),V_{2k+1}(P,-1))$, then
\begin{align*}
p &= \{r(P^{2}+4)U_{2m}(P,-1)+V_{2m}(P,-1)q\}/2\\
&= \{(P^{2}+4)U_{2k+1}(P,-1)U_{2m}(P,-1)+V_{2m}(P,-1)V_{2k+1}(P,-1)\}/2\\
&= V_{2k+2m+1}(P,-1).
\end{align*}

\textbf{Case 2.} If $(r,q)=(-U_{2k+1}(P,-1),V_{2k+1}(P,-1))$, then
\begin{align*}
p &= \{r(P^{2}+4)U_{2m}(P,-1)+V_{2m}(P,-1)q\}/2\\
&= \{-(P^{2}+4)U_{2k+1}(P,-1)U_{2m}(P,-1)+V_{2m}(P,-1)V_{2k+1}(P,-1)\}/2.
\end{align*}
If $2m\geq 2k+1$, then $p=V_{2m-2k-1}(P,-1)$. Otherwise, $p=V_{2k+1-2m}(P,-1)$. Hence, $n=V_{2k+1}(P,-1)V_{2k+2m+1}(P,-1)$ with $V_{2k+1}(P,-1)$ and $V_{2k+2m+1}(P,-1)$ primes; or $V_{2k+1}(P,-1)V_{2m-2k-1}(P,-1)$($m\neq 2k+1$) with $V_{2k+1}(P,-1)$ and $V_{2m-2k-1}(P,-1)$ primes.
\end{proof}

\begin{proof}[Proof of Theorem \ref{th:6}]

Taking $(A,B)=(V_{2m}(P,-1),-V_{2m}^{2}(P,-1)+5)$ in (\ref{eq:lem2.2}), then
$$(2p-V_{2m}(P,-1)q)^{2}-(V_{2m}^{2}(P,-1)-4)q^{2}=-4(V_{2m}^{2}(P,-1)-4).$$
And by Lemma \ref{lem:3}, we have
$$(2p-V_{2m}(P,-1)q)^{2}-(P^{2}+4)U_{2m}^{2}(P,-1)q^{2}=-4(P^{2}+4)U_{2m}^{2}(P,-1).$$
Let $r:=\frac{2p-V_{2m}(P,-1)q}{(P^{2}+4)U_{2m}(P,-1)}$. Then $r$ is an integer satisfying $r^{2}-(P^{2}+4)q^{2}=4$ since $P^{2}+4$ is square-free. From Lemma \ref{lem:3}, we have $(r,q)=(\pm U_{2k}(P,-1),V_{2k}(P,-1))$ for some integer $k$. By (5) and (6) of Lemma \ref{lem:6}, with the process similar to the previous proof, we show that $n=V_{2k}(P,-1)V_{2k+2m}(P,-1)$ with $V_{2k}(P,-1)$ and $V_{2k+2m}(P,-1)$ primes; or $V_{2k}(P,-1)V_{2m-2k}(P,-1)$($m\neq 2k$) with $V_{2k}(P,-1)$ and $V_{2m-2k}(P,-1)$ primes.
\end{proof}

\section{Proof of Theorem 1.7}
\begin{proof}[Proof of Theorem \ref{th:7}]
For the case $\alpha=1$ or the case $q=2$, the theorem is proved in Lemma \ref{lem:7}.

If $n\mid \sigma_{3}(n)$, then
\begin{equation}\label{eq:th:7.0}
\sigma_{3}(n)=(1+p)(1-p+p^{2})(1+q^{3}+\cdots+q^{3\alpha})\equiv 0\pmod{pq^{\alpha}}.
\end{equation}
If $q$ divides $1+p$ and $1-p+p^{2}$ simultaneously, then $q$ is a factor of $(p+1)^{2}-(1-p+p^{2})=3p$. Thus $q=3$.

First of all, we prove Theorem \ref{th:7} for $q=3$ and $\alpha\geq2$. When $\alpha=2$, $p\mid1+3^{3}+9^{3}=757$. Hence, $p=757$. But $1+757^{3}$ can not be divided by $9$. Therefore $\alpha\geq3$. $p\equiv2\pmod{3}$ since $3^{\alpha}$ divides $1+p^{3}$. Moreover, in this case, $q$ divides $1+p$ and $1-p+p^{2}$ simultaneously. However, $1-p+p^{2}$ can never be the multiple of $9$. Thus, $3^{\alpha-1}\mid(1+p)$, and there exists some integer $k_{1}$ such that $p=3^{\alpha-1}k_{1}-1$. Meanwhile, $p$ divides $1+3^{3}+\cdots+3^{3\alpha}=(27^{\alpha+1}-1)/(3^{3}-1)$. Hence, $p$ divides $3^{\alpha+1}-1$ or $3^{2\alpha+2}+3^{\alpha+1}+1$ since $27^{\alpha+1}-1=(3^{\alpha+1}-1)(3^{2\alpha+2}+3^{\alpha+1}+1)$.

If $3^{\alpha+1}-1$ is a multiple of $p$, then so is $3^{\alpha+1}-1-p=3^{\alpha-1}(9-k_{1})$. Therefore $p$ is a prime number not greater than $9$. However this is impossible because $p=3^{\alpha-1}k_{1}-1\geq8$. Hence $p$ divides $3^{2\alpha+2}+3^{\alpha+1}+1$, namely, there exists some integer $k_{2}$ such that
\begin{equation}\label{eq:th:7.1}
3^{2\alpha+2}+3^{\alpha+1}+1=pk_{2}.
\end{equation}
Then $k_{2}\equiv -1\pmod{3^{\alpha-1}}$, that is $k_{2}=3^{\alpha-1}k_{3}-1$ for some integer $k_{3}$. Taking it into (\ref{eq:th:7.1}), we have $$3^{\alpha-1}=\frac{9+k_{1}+k_{3}}{k_{1}k_{3}-81}\geq9.$$ Therefore, $k_{1}\leq92$. And because $p$ divides $3^{2\alpha+2}+3^{\alpha+1}+1$, $p$ also divides $3^{2\alpha+2}+3^{\alpha+1}+1+p(1+3^{\alpha-1}(9+k))=3^{2\alpha-2}(k_{1}^{2}+9k_{1}+81)$. Combining this with $k_{1}\leq92$, we can find no prime $p$ happens to equal $3^{\alpha-1}k_{1}-1$. Thus, no $n=3^{\alpha}p$ ($\alpha>1$) divides $\sigma_{3}(n)$.

Secondly, by contraction, we prove that for $q\equiv2\pmod{3}$ $(q>2)$ and $\alpha>1$, $n\nmid\sigma_{3}(n)$. By (\ref{eq:th:7.0}), $q^{\alpha}$ divides $1+p$ or $1-p+p^{2}$. If $1-p+p^{2}\equiv0\pmod{q}$, then $(2p-1)^{2}\equiv-3\pmod{q}$. Thus $-3$ is a quadratic residue modulo $q$, which is impossible for $q\equiv2\pmod{3}$ $(q>2)$. Otherwise, $q\alpha$ divides $1+p$, i.e., there exists some integer $k_{4}$ such that $p=k_{4}q^{\alpha}-1$.

Moreover, $p$ divides $q^{\alpha+1}-1$ or $q^{2\alpha+2}+q^{\alpha+1}+1$ since $p\mid(1+q^{3}+\cdots+q^{3\alpha})$. If $q^{\alpha+1}-1$ is a multiple of $p$, then so is $q^{\alpha+1}-1-p=q^{\alpha}(q-k_{4})$. So $p<q$, which is impossible since $p=k_{4}q^{\alpha}-1>q$ ($q>2$, $\alpha>1$). Therefore, $p$ divides $q^{2\alpha+2}+q^{\alpha+1}+1$, which means,
\begin{equation}\label{eq:th:7.2}
q^{2\alpha+2}+q^{\alpha+1}+1=pk_{5},
\end{equation}
for some integer $k_{5}$. By this, we have $k_{5}\equiv-1\pmod{q^{\alpha}}$, namely, $k_{5}=k_{6}q^{\alpha}-1$ for some integer $k_{6}$. Taking it back to (\ref{eq:th:7.2}), we obtain that \begin{equation}\label{eq:th:7.3}
q^{2\alpha+2}+q^{\alpha+1}+1=(k_{4}q^{\alpha}-1)(k_{6}q^{\alpha}-1).
\end{equation}
Thus $k_{4}$(or $k_{6}$)$=\frac{q^{\alpha+2}+q+1}{q^{\alpha}-1}$. Moreover $k_{4}$(or $k_{6}$)$\leq q^{2}$.

Let $\alpha=2$. Then
\begin{eqnarray*}
1+q^{3}+q^{6}&=&(q^{2}k_{4}-1)(q^{2}k_{6}-1)\\
q+q^{4}&=&q^{2}k_{4}k_{6}-(k_{4}+k_{6}).
\end{eqnarray*}
Thus, $k_{4}+k_{6}=tq$ for some integer $t$ and $t\leq 2q$ since $k_{4}$, $k_{6}\leq q^{2}$. Replacing $k_{4}+k_{6}$ by $tq$, we have $t\equiv-1\pmod{q}$. Hence, $t=q-1$ or $2q-1$.

If $t=q-1$, then $$1+q^{3}=qk_{4}(q^{2}-q-k_{4})-q+1,$$ $$q^{2}=k_{4}(q^{2}-q-k_{4})-1.$$ Thus if $q\equiv1\pmod{4}$, then $k_{4}^{2}+2\equiv0\pmod{4}$. And if $q\equiv-1\pmod{4}$, then $(k_{4}-1)^{2}+1\equiv0\pmod{4}$. Both cases are impossible.

If $t=2q-1$, then $$1+q^{3}=qk_{4}(2q^{2}-q-k_{4})-2q+1,$$ $$q^{2}=k_{4}(2q^{2}-q-k_{4})-2.$$ Thus if $q\equiv\pm1\pmod{4}$, then $k_{4}(\pm1-k_{4})\equiv3\pmod{4}$, which is impossible.

Let $\alpha\geq3$. Then
\begin{eqnarray*}
1+q^{\alpha+1}+q^{2\alpha+2}&=&(q^{\alpha}k_{4}-1)(q^{\alpha}k_{6}-1)\\
q+q^{\alpha+2}&=&q^{\alpha}k_{4}k_{6}-(k_{4}+k_{6}).
\end{eqnarray*}
Thus, $k_{4}+k_{6}=tq$ for some integer $t$ and $t\leq 2q$ since $k_{4}$, $k_{6}\leq q^{2}$. Replacing $k_{4}+k_{6}$ by $tq$, we have $t\equiv-1\pmod{q}$. Hence, $t=q-1$ or $2q-1$.

If $t=q-1$, then $$1+q^{\alpha+1}=q^{\alpha-1}k_{4}(q^{2}-q-k_{4})-q+1,$$ $$q^{\alpha}=q^{\alpha-2}k_{4}(q^{2}-q-k_{4})-1.$$ Thus $-1\equiv0\pmod{q}$. It is impossible.

If $t=2q-1$, then $$1+q^{\alpha+1}=q^{\alpha-1}k_{4}(2q^{2}-q-k_{4})-2q+1,$$ $$q^{\alpha}=q^{\alpha-2}k_{4}(2q^{2}-q-k_{4})-2.$$ Thus  $-2\equiv0\pmod{q}$ and $q=2$. A contradiction.

All above prove that if $n\mid \sigma_{3}(n)$ with $n=pq^{\alpha}$ and $q\not\equiv1\pmod{3}$, then $n$ is an even perfect number except for 28.

\end{proof}

\subsection*{Acknowledgments}
This work is supported by the National Natural Science Foundation of China (Grant No. 11501052 and Grant No. 11571303).

\bibliographystyle{amsplain}

\begin{thebibliography}{9}

\bibitem{Cai}
T. Cai, The book of numbers, World Scientific Publishing Company, 2016.

\bibitem{Cai2013}
T. Cai, D. Chen and Y. Zhang, Perfect numbers and Fibonacci primes, \textit{Int. J. Number Theory} \textbf{11} (2015), no. 1, 159--169.

\bibitem{Cai2014}
T. Cai, L. Wang and Y. Zhang, Perfect numbers and Fibonacci primes (II), Preprint, arXiv:1406.5684v1, 7 pp.

\bibitem{Jones2003}
J. P. Jones, Representation of solutions of Pell equations using Lucas sequences, \textit{Acta Acad.paedagog.agriensis Sect.mat} \textbf{30} (2003), 75-86.

\bibitem{Matiyasevich1970}
Y. V. Matiyasevich, Enumerable sets are diophantine, \textit{Dokl. Akad. Nauk SSSR}, \textbf{191} (1970), 279--282.

\bibitem{Zhong2016}
H. Zhong and T. Cai, On the Lucas property of linear recurrent sequences, \textit{Int. J. Number Theory}, \textbf{13} (2017), no. 6, 1617-1625.


\end{thebibliography}

\end{document}